\documentclass[reqno,12pt]{amsart}
\usepackage{amsmath,amssymb,amsthm,mathtools}
\usepackage{mathrsfs}
\usepackage{bm}
\usepackage{tikz-cd}
\usepackage{booktabs}
\usepackage{xcolor}
\usepackage{hyperref}
\usepackage{float}
\usepackage{orcidlink}
\usepackage{comment}
\hypersetup{
    colorlinks=true,
    linkcolor=blue,
    citecolor=blue,
    urlcolor=blue}
\usepackage[a4paper,margin=1in]{geometry}
\newtheorem{theorem}{Theorem}[section]
\newtheorem{lemma}[theorem]{Lemma}
\newtheorem{proposition}[theorem]{Proposition}
\newtheorem{corollary}[theorem]{Corollary}

\theoremstyle{definition}

\newtheorem{remark}[theorem]{Remark}

\allowdisplaybreaks

\title{On the Diophantine Equation $\displaystyle p^x+ (2p+1)^y =z^2$ with Consecutive Exponents}

\author{Subhasis Panda} 

\address{Department of Mathematics, School of Advanced Sciences, VIT-AP University, Amravati 522241, Andhra Pradesh, India}

\email{subhasispanda559@gmail.com}



\begin{document}

\begin{abstract} { We determine all positive integer solutions to
		$p^x+(2p+1)^y=z^2$, where $p$ is an odd prime and $|x-y|=1$. The unique solution is $(x,y,z)=(2,1,p+1)$ except when $p=7$, for which there is one additional solution, $(x,y,z)=(4,3,76)$. The proof reduces the nontrivial factorization case to an auxiliary equation, bounds its exponents using an AM--GM inequality, and isolates the exceptional solution with a divisibility argument modulo $2k^2$.

}\end{abstract}

\maketitle

\setcounter{tocdepth}{1}
\tableofcontents

\section{Introduction}
Diophantine equations of the form $a^x+b^y=z^2$ and $a^x-b^y=z^2$,
where $a$ and $b$ are fixed positive integers, have received a lot of attention \cite{Le2003,Le2005,Sroysang2012,MinaBacani2021}. In most cases, the integers a and b are fixed, and the aim is to determine all positive integer solutions $(x,y,z)$ of the corresponding equation.

For example, Sroysang~\cite{Sroysang2012} studied the equations $3^x + 5^y = z^2$. 
Suvarnamani \cite{Suvarnamani2014} studied the equation $p^x+(p+1)^y=z^2$, where $p$ is an odd prime, and proved that it has a unique positive integer solution when $p=3$. Subsequently, Choudhury and Sarma \cite{ChoudhurySarma2026} investigated the  equations $(p+a)^x+p^y=z^{2}$ and $(p+a)^x-p^y=z^{2}$, where $p$ and $p+a$ are odd primes.
Later, Mina and Bacani \cite{MinaBacani2021} considered the more general family $p^x+(p+4k)^y=z^2$, where both $p$ and $p+4k$ are prime, and  obtained necessary conditions for the existence of positive integer solutions. Recently, Gaha and Mezroui \cite{GahaMezroui2026} studied the equation $43^x+87^y=z^2$, under the condition that $x$ and $y$ are consecutive integers. In this paper, we investigate the general family
\begin{equation} \label{equation1}
	p^x +  ( 2p +1 )^y = z^2    
\end{equation}
where $p$ is an odd prime and $|x-y|=1$, without assuming that $2p+1$ is prime. 
The condition $|x-y|=1$ naturally divides the problem into two cases according to whether $x=y+1$ or $y=x+1$.  
We show that $(x,y,z)=(2,1,p+1)$ is always a solution, and that any further solution forces $2p+1$ to be composite and reduces to an auxiliary equation, which we solve completely. Our main result is the following.

\begin{theorem}\label{T1}
	Let $p$ be an odd prime and suppose that
	\begin{equation*}
		p^x + \left ( 2p +1 \right)^y = z^2,
	\end{equation*}
	where $x$, $y$, $z$ are positive integers with $|x-y|=1$. Then exactly one of the following holds.
	\begin{enumerate}
		\item[(i)] If $p\neq 7$, then the only solution is $(x,y,z)=(2,1,p+1)$.
		\item[(ii)]  If $p=7$, then the only solutions are $(x,y,z)=(2,1,8) \,\, \text{and} \,\, (4,3,76)$.
	\end{enumerate}
\end{theorem}

\section{Preliminary Results}
We prove the main theorem using elementary number theory. In this section, we first state the Lifting-the-Exponent Lemma (LTE), which will be useful in our proofs. We then establish several preliminary results that will be used in the proofs of our main theorem.
\begin{lemma} {\normalfont\cite{Heuberger2020}} \label{lemma3}
	Let $p$ be an odd prime, let $n$ be a positive integer, and let $u,w\in\mathbb{Z}$ satisfy $p\mid (u-w)$ and $p\nmid uw$. Then
	\begin{equation*}
		v_p(u^n-w^n)=v_p(u-w)+v_p(n).
	\end{equation*}
\end{lemma}
In this paper, we apply Lemma~\ref{lemma3} with $u=q=2p+1$ and $w=1$. Since $v_p(q-1)=v_p(2p)=1$, we immediately obtain the following consequence.

\begin{corollary} \label{corollary1}
	For every integer $b \geq 0$, we have 
	\begin{equation*}
		v_p\bigl((2p+1)^{2b+1}-1\bigr)=1+v_p(2b+1).
	\end{equation*}
\end{corollary}
\begin{proof}
	We apply Lemma~\ref{lemma3} with $u=2p+1$, $w=1$, and $n=2b+1$. Since $p\mid 2p=u-w$, $p\nmid u w$, and $v_p(u-w)=v_p(2p)=1$, the result follows immediately.
\end{proof}

Throughout the paper, let $p$ be an odd prime, $q:=2p+1$, and
$x,y,z$ be positive integers that satisfy $\displaystyle p^{x}+q^{y}=z^{2}$ with $|x-y|=1.$ 
Since $x$ and $y$ are consecutive integers, they have opposite parity. Therefore, exactly one of $x$ and $y$ is even and the other is odd. 
\subsection{Case A: x odd, y even}
In this case, we write $x=2a+1$ and $y=2b$, where $a\geq 0$ and $b\geq 1$ (since $y$ is positive). The condition \(|x-y|=1\) then gives rise to exactly two sub-cases: $y=x-1$, which is equivalent to $b=a$, and $y=x+1$, which is equivalent to $b=a+1$. 
In either sub-case, the equation $p^x+q^y=z^2$ takes the form $p^{2a+1}+q^{2b}=z^2,$ or equivalently, $p^{2a+1}=z^2-q^{2b}=(z-q^b)(z+q^b)$. We first show that the two factors on the right-hand side are coprime. 
Let $d=\gcd(z-q^b,\,z+q^b)$. Then $d\mid\bigl((z+q^b)-(z-q^b)\bigr)=2q^b$ and $d\mid p^{2a+1}$. 
Since $p$ is odd, $d$ is also odd. Hence $d\mid q^b$. Therefore,
$d\mid p^{2a+1} \,\, \text{and} \,\, d\mid q^b$. Since $\gcd(p,q)=1$, we conclude that $d=1$. 
Thus, $\gcd(z-q^b,\,z+q^b)=1$. Since $z^2-q^{2b}=(z-q^b)(z+q^b)$ is a coprime factorization of a prime power, one of the two factors must be $1$.
As $z+q^b > z-q^b >0$ it follows that $z-q^b=1$ and $z+q^b=p^{2a+1}$. 
By subtracting the two equations, we obtain the following. 
\begin{equation}\label{equation2}
	p^{2a+1}-1=2q^b   
\end{equation}
First, we consider the sub-case of $y=x+1$, (equivalently, $b=a+1$), which gives us the equation $p^{2a+1}-1=2q^{a+1}$. The following lemma shows that this equation has no solution.
\begin{lemma} \label{lemma1}
	Let $p$ be an odd prime, and $q=2p+1$. Then equation 
	\begin{equation*}
		p^{2a+1}-1=2q^{a+1}  
	\end{equation*}
	has no solution for $a  \geq 0$.
\end{lemma}
\begin{proof}
	Suppose, to the contrary, that $p^{2a+1}-1=2q^{a+1}$. Now, reducing the above equation modulo $p$, we obtain $-1\equiv 2q^{a+1} \equiv 2 \pmod p$. Thus $p\mid3$. 
	As $p$ is an odd prime, we must have $p=3$, and consequently $q=7$. The equation becomes $3^{2a+1}-1=2\cdot7^{a+1}$. 
	Now, reducing this equation modulo $7$, we obtain $3^{2a+1}-1 \equiv 0 \pmod 7$.
	Since the multiplicative order of $3$ modulo $7$ is $6$, this gives $6 | (2a+1)$. This is impossible because $2a+1$ is odd. Hence no such $a$ exists.
\end{proof}

Next, we consider the sub-case $(y=x-1)$, (equivalently, $b=a$), with $a \ge 1$.
In this case, the equation~(\ref{equation2}) becomes $p^{2a+1}-1=2q^{a}$.
We show that the equation~(\ref{equation2}) has no solution for this sub-case as well.
\begin{lemma} \label{lemma2}
	
	Let $p$ be an odd prime and let $q=2p+1$. Then the equation
	\begin{equation*}
		p^{2a+1}-1=2q^a   
	\end{equation*}
	has no solution for $a \ge 1$.    
\end{lemma}

\begin{proof}
	The proof follows exactly the same argument as that of Lemma~\ref{lemma1}. Reducing the equation modulo $p$ gives $p=3$, and hence $q=7$. 
	Next, reducing this equation modulo 7, we obtain $3^{2a+1}\equiv1\pmod7$. Since the order of $3$ modulo $7$ is $6$, it follows that $6\mid(2a+1)$, which is impossible.     
\end{proof}
Hence the equation~(\ref{equation1}) has no solution for case A. Note that if $a=0$ then $x=1$ and $y=0$. 
The equation $ p^{2a+1}-1=2q^a$ becomes $p-1=2$ and and hence $p=3$. 
Therefore, the  equation~(\ref{equation1})  has the solution $(p,x,y,z)=(3,1,0,2)$. Since $y$ is assumed to be positive in Theorem~\ref{T1}, this case is excluded from the proof Lemma~\ref{lemma2}.

\subsection{Case B: x even, y odd} In this case, we write $x=2a$ and $y=2b+1$, where $a \ge1 $ and $b \ge 0$.  Then
\begin{equation*}
	q^{2b+1} = z^2-p^{2a} = (z-p^a)(z+p^a).    
\end{equation*}
We now consider two cases. If one of the factors is equal to $1$, we call it the \emph{trivial split}. 
Otherwise, both factors are greater than $1$, and we call it the \emph{non-trivial split}. We first discuss the trivial split.

In case of a trivial split, we have $z-p^a=1$ and $z+p^a=q^{2b+1}$. Subtracting the first equation from the second, we obtain 
\begin{equation} \label{equation3}
	q^{2b+1}-1=2p^a.   
\end{equation}
Since $x=2a$, $y=2b+1$, and $|x-y|=1$, there are two sub-cases:
\begin{enumerate}
	\item[(i)] $y=x-1$, which is equivalent to $a=b+1$;
	\item[(ii)] $y=x+1$, which is equivalent to $a=b$.
\end{enumerate}
For sub-case (i), the equation~(\ref{equation3}) becomes $(2p+1)^{2b+1}-1=2p^{\,b+1}$. The equation admits a solution for $b=0$ corresponding to $(x,y,z)=(2,1,p+1)$. The following proposition shows that this is the only solution.  
\begin{proposition}\label{P2.5}
	For every odd prime $p$, the equation $(2p+1)^{2b+1}-1=2p^{\,b+1}$
	has a unique solution $b=0$,  giving $(x,y,z)=(2,1,p+1)$.
\end{proposition}

\begin{proof}
	If $b=0$, then the equation becomes $(2p+1)-1=2p$, which holds. This gives the solution $(x,y,z)=(2,1,p+1)$ because $p^2+(2p+1)=(p+1)^2$. 
	Now suppose that $b\ge1$. Since $q=2p+1\equiv1\pmod p$ and $v_p(q-1)=v_p(2p)=1$, from Corollary~(\ref{corollary1}), we obtain $v_p(q^{2b+1}-1)=1+v_p(2b+1).$ 
	By comparing the $p$-adic valuations of both sides of $q^{2b+1}-1=2p^{\,b+1}$, we obtain $1+v_p(2b+1)=b+1$, or equivalently, $v_p(2b+1)=b$. 
	Hence $p^b\mid(2b+1)$ which implies $p^b\le2b+1$. If $b=1$, then $p\le3$. 
	Since $p$ is an odd prime, we must have $p=3$, but $7^3-1=342\neq18=2\cdot3^2$, which is impossible. 
	Next, for $b \ge 2$,  we claim that $p^b>2b+1$. This follows by induction on $b$. The statement is true for $b=2$, since $p^2\ge9>5$. 
	Assume that $p^b>2b+1$ for some $b\ge2$. Then
	\begin{equation*}
		p^{b+1}\ge3p^b>3(2b+1)>2(b+1)+1.
	\end{equation*}
	Therefore the equation $(2p+1)^{2b+1}-1=2p^{\,b+1}$ has no solution for $b \geq 1$ and a unique solution at $b=0$.    
\end{proof}

We now consider sub-case (ii), where $a=b$. In this case, equation~(\ref{equation3}) becomes $(2p+1)^{2b+1}-1=2p^{b}$. 
This sub-case also leads to no solution of equation~(\ref{equation3}).

\begin{lemma}
	For every odd prime $p$, the equation $q^{2b+1}-1 = 2p^b$ has no solution for any integer $b\ge0$.
\end{lemma}
\begin{proof}
	If $b=0$, then the equation becomes $2p=2$, which is impossible. Now suppose that $b\ge1$. Since $2b+1\ge3$, we have  
	\begin{equation*}
		(2p+1)^{2b+1}-1\ge(2p+1)^3-1=8p^3+12p^2+6p.    
	\end{equation*}
	Moreover,
	\begin{equation*}
		8p^3+12p^2+6p-2p^b = p\left(8p^2+12p+6-2p^{\,b-1}\right).
	\end{equation*}
	If $b=1$, then $8p^2+12p+6-2=8p^2+12p+4>0$. Now, let $b\ge2$. Since
	\begin{equation*}
		(2p+1)^{2b+1}\ge(2p+1)^3(2p+1)^{2b-2} \,\, \text{and} \,\,  (2p+1)^2>p,   
	\end{equation*}
	it follows that $(2p+1)^{2b+1}>(2p+1)^3p^{b-1} > 2p \cdot p^{b-1} = 2 p^b $. Therefore $(2p+1)^{2b+1}-1>2p^b$ and hence the equation has no solution.
\end{proof}
The trivial split produces only the solution $(x,y,z)=(2,1,p+1)$. Next, we consider non-trivial split case. 
In this case, both factors in $q^{2b+1}=(z-p^a)(z+p^a)$ are greater than $1$. Since $\gcd(z-p^a,z+p^a)=1$, every prime power dividing $q^{2b+1}$ must be assigned entirely to one of the two factors.
Thus there exist coprime integers $r$ and $s$ satisfying $q=rs$, with $1<r<s$,
such that $z-p^a=r^{\,2b+1}$ and $z+p^a=s^{\,2b+1}$. 
The difference of these two equations gives $s^{\,2b+1}-r^{\,2b+1}=2p^a$. 
Now, from the factorization of $s^{\,2b+1}-r^{\,2b+1}$, we obtain 
\begin{equation}\label{eq4}
	(s-r)\sum_{i=0}^{2b}s^{\,2b-i}r^i=2p^a.
\end{equation}
For our convenience, let $\displaystyle S=\sum_{i=0}^{2b}s^{\,2b-i}r^i$. Then equation~\eqref{eq4} becomes $(s-r)S=2p^a$. Next, we want to determine the possible values of $s - r$.
\begin{lemma} \label{lemma5}
	The gcd of $s-r$ and $S$ is $1$.    
\end{lemma}

\begin{proof}
	Let $d=\gcd(s-r,S)$. Since $s\equiv r \pmod{s-r}$,  we have $S\equiv\sum_{i=0}^{2b}r^{2b} =(2b+1)r^{2b} \pmod{s-r}$. 
	Hence $d\mid(2b+1)r^{2b}$. Since $\gcd(r,s-r)=\gcd(r,s)=1$, it follows that
	$\gcd(r^{2b},s-r)=1$. Therefore, $d\mid(2b+1)$. 
	
	As $d \mid (s-r) $ and $d\mid S$, we have $d^{\,2} \mid (s-r) S = 2 p^a$. In addition, $d$ is odd, and it follows that $d^{\,2}\mid p^a$.
	Therefore, $d=p^t$ for some integer $t\ge0$. If $d>1$, then $d\ge p$, and since $d\mid(s-r)$, we obtain $s-r\ge d\ge p$.  But $r \ge 3$, as $r \mid q$ and $r >1$, so we have
	\begin{equation*}
		s=\frac{2p+1}{r}\le\frac{2p+1}{3}<p.
	\end{equation*} 
	Therefore $s-r<s<p\le d\le s-r$, which is impossible. Hence $d=1$.
\end{proof}

The case $b=0$ is impossible. Indeed, if $b=0$, then $S=1$, and equation~\eqref{eq4} gives $s-r=2p^a\geq 2p$. However, since $r\geq 3$, we have $s=\frac{2p+1}{r}\leq\frac{2p+1}{3}<p$, which implies $s-r<p$, a contradiction. Thus $b\geq 1$, so $S\geq 3$.

Now, since $(s-r)S=2p^a$ and $\gcd(s-r,S)=1$ by Lemma~\ref{lemma5}, the two factors cannot share any prime factor. As their product is $2p^a$ and $S\geq 3$, we must have either $s-r=1$ and $S=2p^a$, or $s-r=2$ and $S=p^a$.

Therefore $s-r\in\{1,2\}$. 
Now, suppose $s-r=1$, equivalently, $s=r+1$. since $rs=2p+1$, we obtain $p = \frac{r(r+1)-1}{2}$.
However, $r(r+1)$ is even, so $r(r+1)-1$ is odd. Hence $\frac{r(r+1)-1}{2}$ is not an integer, which is a contradiction. Therefore, $s-r\neq1$ and hence $s-r=2$.

Since both r and s are divisors of an odd integer and $s-r=2$,  we can write $r=2k-1$ and $s=2k+1$ for some $k \geq 2$. From $rs=2p+1$, we obtain $(2k-1)(2k+1)=2p+1$, which simplifies to $4k^2-1=2p+1$. 
Hence $p=2k^2-1$. Thus every non-trivial solution must satisfy $2p+1=(2k-1)(2k+1)$, where $p=2k^2-1$ is prime.

\begin{remark}\label{R2.8}
	In the sub-case $a=b$ (equivalently, $y=x+1$), the sum $S= \displaystyle \sum_{i=0}^{2b}s^{2b-i}r^i$ contains the term $s^br^b=(rs)^b=(2p+1)^b>p^b$. Hence $S>(2p+1)^b>p^b=p^a$,
	which is impossible. Therefore the sub-case $a=b$ has no solution.
\end{remark}

Therefore, the only remaining subcase is $a=b+1$. Now, if we substitute $r=2k-1$, $s=2k+1$ and $p=2k^2-1$ into the equation $s^{\,2b+1}-r^{\,2b+1}=2p^a$, then we obtain
\begin{equation*}
	(2k+1)^{\,2b+1}-(2k-1)^{\,2b+1} = 2(2k^2-1)^{\,b+1}	
\end{equation*}

\section{Proof of the main result}
Using the results established above, we now prove the main theorem. The proof is divided into the trivial and non-trivial factorization cases, leading to the two possibilities in Theorem~\ref{T1}. \\

\noindent \textbf{Proof of Theorem~\ref{T1}:} 
By Lemmas~\ref{lemma1} and~\ref{lemma2}, the equation~\eqref{equation1} has no solution for case A. By Proposition~\ref{P2.5}, the equation has the solution $(x,y,z)=(2,1,p+1)$. The trivial split produces only the solution $(x,y,z)=(2,1,p+1)$. If this is not the case, then Remark~\ref{R2.8} shows that $a=b$ is not possible. Hence $a=b+1$. In this case, by the factorization obtained above, we have $2p+1=(2k-1)(2k+1)$ , where $p=2k^2-1$. This reduces to
\begin{equation}\label{e5}
	(2k+1)^{2b+1}-(2k-1)^{2b+1}
	=2(2k^2-1)^{b+1}.
\end{equation}

From the identity $X^n-Y^n= \displaystyle (X-Y)\sum_{i=0}^{n-1}X^{n-1-i}Y^i$,
equation~\eqref{e5} gives
\begin{equation*}
	\sum_{i=0}^{2b}(2k+1)^{2b-i}(2k-1)^i
	=(2k^2-1)^{b+1}.	
\end{equation*}
For each $0\le i\le b-1$, pair the $i$-th and $(2b-i)$-th
terms. By the arithmetic-geometric mean inequality, we obtain
\begin{equation*}
	(2k+1)^{2b-i}(2k-1)^i
	+(2k+1)^i(2k-1)^{2b-i}
	\ge
	2(2k+1)^b(2k-1)^b.
\end{equation*}
Since there are total $b$ such pairs along with the middle term $(2k+1)^b(2k-1)^b$, we obtain 

\begin{equation*} 
	\sum_{i=0}^{2b}(2k+1)^{2b-i}(2k-1)^i \ge (2b+1)(2k+1)^b(2k-1)^b. 
\end{equation*}

Now combining this with~\eqref{e5}, we obtain $(2k^2-1)^{b+1}\ge
(2b+1)(4k^2-1)^b$. Since 
\begin{equation*}
	\frac{4k^2-1}{2k^2-1}
	=2+\frac{1}{2k^2-1}>2,
\end{equation*}
it follows that
\begin{equation}\label{e6}
	2k^2-1>(2b+1)2^b.
\end{equation}
Next, we show that the auxiliary equation in \eqref{e5} has only one solution. Now, using the binomial expansion, we have
\begin{equation}\label{eq8}
	\frac{(2k+1)^{2b+1}-(2k-1)^{2b+1}}{2} = 1+\sum_{j=1}^{b}
	\binom{2b+1}{2j-1}(2k)^{2b+2-2j}.
\end{equation}
For every $1\le j\le b$, the exponent $2b+2-2j$ is at least $2$. Thus, all terms except the constant term are divisible by $2k^2$. 
Hence equation~(\ref{eq8}) is congruent to $1$ modulo $2k^2$.
On the other hand, $(2k^2-1)^{b+1}\equiv(-1)^{b+1}\pmod{2k^2}$. Therefore,
$1\equiv(-1)^{b+1}\pmod{2k^2}$. Since $k\ge2$, we have $b$ is odd and hence $b+1$ is even.
Expanding the right hand side of equation~\eqref{e5} in powers of $2k^2$, we obtain $(2k^2-1)^{b+1} = 1-2(b+1)k^2+4k^4D$ for some integer $D$. 
On the other hand, from equation~(\ref{eq8}), we have
\begin{equation*}
	\frac{(2k+1)^{2b+1}-(2k-1)^{2b+1}}{2} = 1+4b(2b+1)k^2+16k^4A
\end{equation*}
for some integer $A$. 
Now comparing these two expressions using equation~\eqref{e5}, we obtain $4b(2b+1)k^2+16k^4A = -2(b+1)k^2+4k^4D$. 
Now, dividing both sides $2k^2$ and simplifying, we get
$4b^2+3b+1=2k^2(D-4A)$. Therefore, $2k^2\mid 4b^2+3b+1$. In particular, we have
\begin{equation} \label{eq7}
	2k^2\le 4b^2+3b+1.
\end{equation}
If $b\ge3$, then we have $2^b\ge 2b+2$. Therefore, form equation~\eqref{e6}, we get 
\begin{equation*}
	2k^2-1>(2b+1)2^b
	\ge (2b+1)(2b+2)
\end{equation*}
Hence $2k^2>4b^2+6b+3$, which contradicts~\eqref{eq7}. Thus $b<3$. Since $b$ is a positive odd integer, we have $b=1$. For $b=1$, the condition $2k^2\mid 4b^2+3b+1$ gives $2k^2\mid 8$. Since $k\ge2$, we obtain $k=2$. Thus $(k,b)=(2,1)$. Consequently, $p=2k^2-1=7$
and since $a=b+1$, we have $(a,b)=(2,1)$. Hence $(x,y)=(4,3)$.
Finally, $z^2=7^4+15^3=5776=76^2$, so $z=76$. Therefore,
$(p,x,y,z)=(7,4,3,76)$. This completes the proof.

\section*{Declarations}

\subsection*{Funding}
Not applicable.

\subsection*{Disclosure of Interest}

No potential conflict of interest was reported by the author.

\bibliographystyle{amsplain}
\bibliography{mybib}

\end{document}